\documentclass[12pt]{article}
\usepackage{amsthm,amsfonts, amsbsy, amssymb,amsmath,graphicx}
\usepackage{graphics}
\usepackage{cite}

\def\thtext#1{
\catcode`@=11
\gdef\@thmcountersep{. #1}
\catcode`@=12}

\long\def\notes#1#2{
\begin{table}[b]
\vspace{#1pt}
\rule{3truecm}{0.3pt}
\vskip4pt
\par\noindent
#2
\end{table}}

\def\thank#1{\parbox{\hsize}{
\par \tolerance=200 \parindent=10pt \small #1.}
\vskip1pt \par\noindent}

\def\subclass#1{\parbox{\hsize}{
\tolerance=200 \parindent=10pt
{\small\it 2010 Mathematical Subject Classification}.
\small #1.}
\vskip1pt \par\noindent}

\def\keywords#1{\parbox{\hsize}{
\tolerance=200 \parindent=10pt
{\it Key words and phrases}. \small #1.} \vskip1pt}

\newtheorem{theorem}{Theorem}[section]
\newtheorem{lemma}{Lemma}[section]

\newcounter{Rk}[section]
\renewcommand{\thtext}{\thesection.\arabic{Rk}}
\newenvironment{remark}{\trivlist\item[\hskip\labelsep{\bf Remark}]
\refstepcounter{Rk}{\bf\thesection.\arabic{Rk}.}}%
{\endtrivlist}

\newcounter{Df}[section]
\renewcommand{\thtext}{\thesection.\arabic{Df}}
\newenvironment{definition}{\trivlist\item[\hskip\labelsep{\bf Definition}]
\par\refstepcounter{Df}{\bf\thesection.\arabic{Df}.}}%
{\endtrivlist}

\newcounter{Ex}[section]
\renewcommand{\thtext}{\thesection.\arabic{Ex}}
{\endtrivlist}

\newcommand{\biqbrdef}{\raisebox{-0.\height}{\centering\includegraphics[width=8cm]{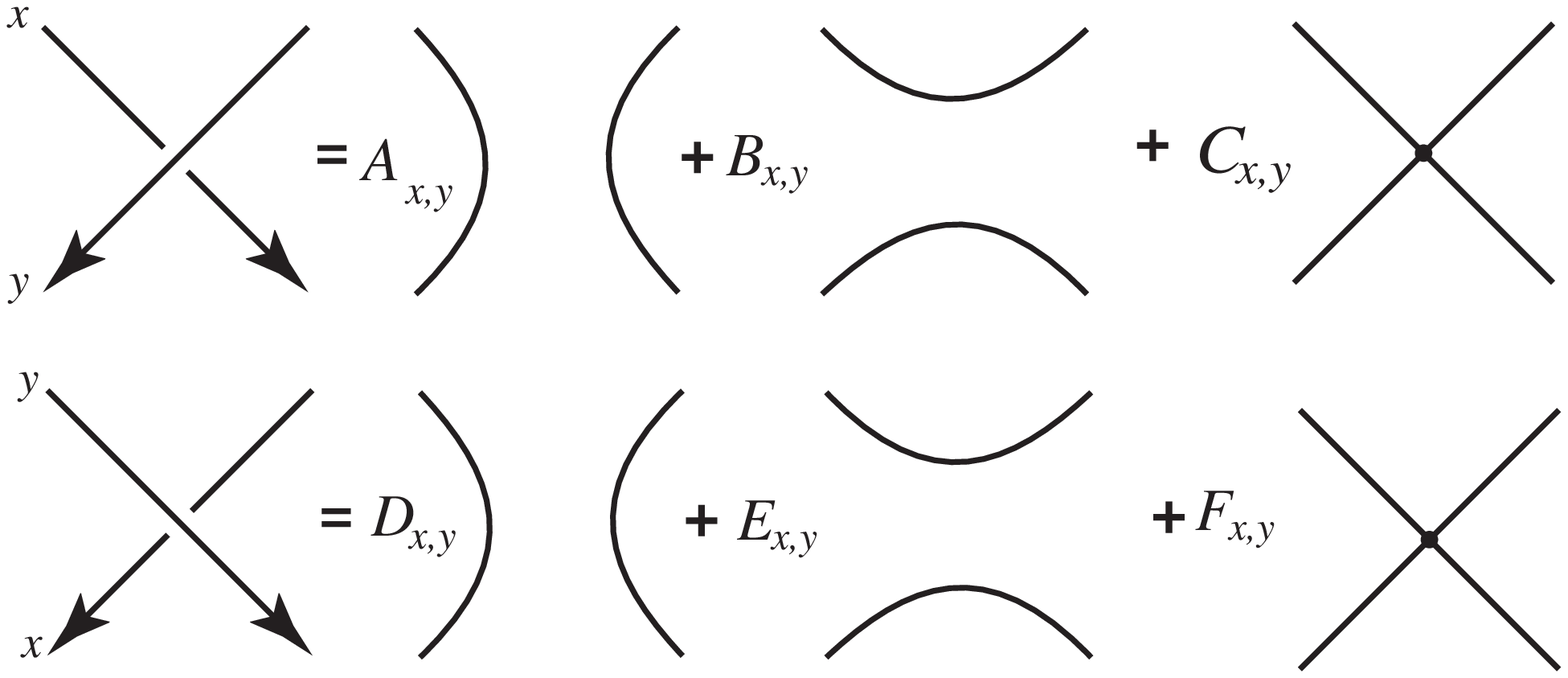}}}
\newcommand{\biqcol}{\raisebox{-0.\height}{\centering\includegraphics[width=6cm]{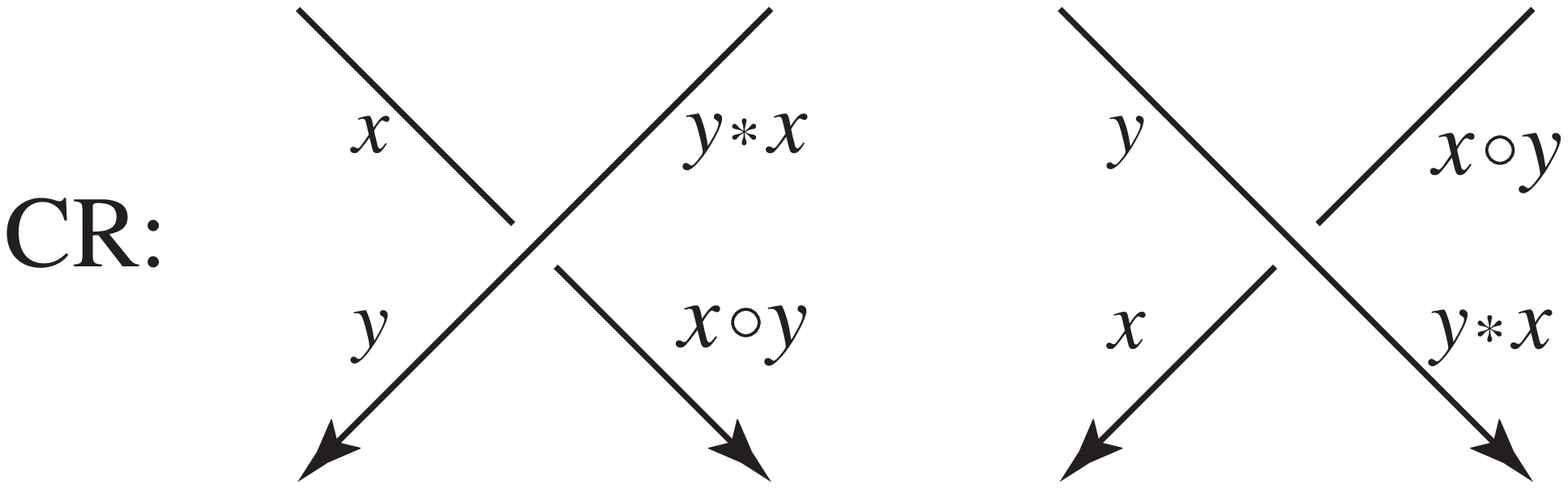}}}

\title{Picture-valued parity-biquandle bracket II. Examples}

\author{Denis P. Ilyutko, Vassily O. Manturov}

\begin{document}
\date{}

\maketitle

\abstract{In~\cite{IM_bq} we constructed the parity-biquandle bracket valued in {\em pictures} (linear combinations of $4$-valent graphs).
We gave no example of classical links such that the parity-biquandle bracket of which is not trivial.

In the present paper we slightly change the notation of the parity-biquandle bracket and give examples of knots and links having a non-trivial parity-biquandle bracket. As a result we get the minimality theorem.

This is the first evidence that graphs (link shadows) appear as invariants of link diagrams instead of just polynomials groups and other tractable objects.}

\notes{0}{
\subclass{57M15, 57M25, 57M27} 
\thank{The first named author was partially supported by grants of RF President NSh -- 6399.2018.1 and RFFI 19--01--00775. The second
named author was supported by the Laboratory of Topology and Dynamics, Novosibirsk State University (grant No. 14.Y26.31.0025 of
the government of the Russian Federation)}}


 \section{Introduction}


Virtual knots, free knots (we omit the over/under information at crossings and replace the cyclic ordering with the cross structure for virtual knots) and other knots have a very important feature which turns out to be trivial for classical knots (but not links): the {\em parity}. The simplest (Gaussian) parity is defined in terms of Gauss diagrams: a chord is {\em even} if it is linked with evenly many chords, otherwise it is odd. Crossings are called {\em even} ({\em odd}) respectively to the chords. The parity allows one to realise the following principle~\cite{IMN2,IMN,IMN3,Mant31,Mant32,Mant33,Mant35,Mant39,Mant36,Mant38,Mant40,Mant41,Mant42,ManIly}:\\
If a knot diagram is complicated enough then it realises itself. The latter means that it appears as a subdiagram in any
diagram equivalent to it. This principle comes from a very easy formula $[K]=\widetilde{K}$, where $K$ on the LHS is a (virtual or free) knot (i.e., a diagram considered up to various moves), and $\widetilde{K}$ on the RHS is a single diagram (the underlying graph) of the knot (in the case of free knots $\widetilde{K}=K$), which is complicated enough and considered as an element of a linear space formally generated by such diagrams.

The bracket $[\cdot]$ is a {\em diagram-valued} invariant of (virtual, free) knots~\cite{IMN2,IMN,IMN3,Mant31,Mant32,Mant33,Mant35,Mant39,Mant36,Mant38,Mant40,Mant41,Mant42,ManIly}. It is important for us to know that
 \begin{enumerate}
  \item[1)]
it is defined by using {\em states} in a way similar to the Kauffman bracket,
  \item[2)]
it is valued not in numbers or (Laurent) polynomials but in {\em diagrams} meaning that we do not completely resolve a knot diagram
leaving some crossings intact.
 \end{enumerate}
It is important to note that for some (completely odd) diagrams, no crossings are smoothed at all.

It is the first appearance of diagram-valued invariants in knot theory. For virtual knots, it allows one to make very strong conclusions about the shape of any diagram by looking just at one diagram. Unfortunately, the value of the bracket at classical knots is trivial.

One of the simplest knot invariants is the colouring invariant: one colours edges of a knot diagram by colours from a given palette and
counts some colourings which are called {\em admissible}. Colouring invariants are a partial case of quandle invariants, which, in turn,
are partial cases of biquandle invariants.

The beautiful idea due to Nelson, Orrison, Rivera~\cite{NOR,NR} allows one to use colours or (bi)quandles in order to enhance various invariants of knots. To this end, one takes a colouring of a knot diagram and a quantum invariant which satisfies a certain skein-relation (say, Kauffman
bracket) and tries to take different coefficients for the states: we take into account the colours of the (short) arcs of the knot diagram. Then one considers a sum over all admissible colourings. Hence, (bi)quandle colourings, though usually not being ``invariants of crossings'', turn out to be a very powerful tool for enhancing knot invariants.

Note that parity can be treated in terms of biquandle colourings. Namely, there is a very simple biquandle which allows one to say whether a crossing is even or odd by looking at colours of edges incident to this crossing.

In~\cite{IM_bq} we brought together the ideas from parity theory~\cite{IMN2,IMN,IMN3,Mant31,Mant32,Mant33,Mant35,Mant39,Mant36,Mant38,Mant40,Mant41,Mant42,ManIly} and the ideas from~\cite{NOR} and constructed the universal biquandle picture-valued invariant of classical and virtual links. We axiomatised the new invariant in such a way that it {\em a priori} dominated both the biquandle bracket and the parity bracket, thus being a very strong invariant of both virtual and classical
links: it is at least as strong as the biquandle invariant for classical knots and at least as strong as the parity bracket for
virtual knots. Unlike Nelson, Orrison and Rivera we can have ``pictures'' as values of our invariant. Note that we did not know any partial examples of invariants for classical links, which were {\em diagram-valued} and non-trivial. In this paper we slightly change the notation of the parity-biquandle bracket and give examples of a classical knot and link having a non-trivial parity-biquandle bracket. As a result we get the minimality theorem.

The paper is organised as follows. In the next section we give the main definitions concerning knots and biquandles. In Sec.~\ref{sec:parbiqbr} we present a construction of a little modified version of the parity-biquandle bracket from~\cite{IM_bq} and give some minimality theorems. In the final section we give examples of a classical link and knot the parity-biquandle bracket at which is non-trivial.

\section{Knots and biquandle}\label{sec:kn&biq&par}

 \subsection{Knots}

Throughout the paper we consider a knot from combinatorial point of view.

 \begin{definition}
A {\em $4$-valent graph} is a finite graph with each vertex having degree four.

A $4$-valent graph is called a {\em graph with a cross structure} or a {\em framed $4$-graph} if for every vertex the four emanating
half-edges are split into two pairs of half-edges (we have the structure of opposite edges). Half-edges constituting a pair are called {\em opposite}.

A $4$-valent graph is called a {\em graph with a cycle order} if for every vertex the four emanating half-edges are cyclically ordered.

In both cases we also admit diagrams which consist of disjoint unions of the above-mentioned diagrams and several circles with no
vertices.
 \end{definition}

 \begin{remark}
Every $4$-valent graph with a cycle order is a graph with the cross structure.
 \end{remark}

 \begin{definition}
By a ({\em classical}) {\em link diagram} we mean a plane $4$-valent graph with a cycle order and over/undercrossing structure at each vertex. Half-edges constituting a pair are called {\em opposite}. A diagram is called {\em oriented} if each edge is oriented and opposite edges has the same orientation, i.e., the one is oriented to the vertex and the other is from the vertex. The relation of half-edges to be opposite allows one to define the notion of unicursal component and to count the number of unicursal components of a link diagram, see Fig.~\ref{knlivir} for classical knots (links with one unicursal component) and links. Vertices of a diagram are called {\em crossings}.

A {\em virtual link diagram} is a generic immersion of a $4$-valent graph into the plane such that we have a cycle order and over/undercrossing structure at each vertex and mark each edge intersection by a circle, see Fig.~\ref{knlivir}. For virtual links the definitions of components and of an oriented virtual link are the same as for classical links.
 \end{definition}

 \begin{figure}
  \centering\includegraphics[width=280pt]{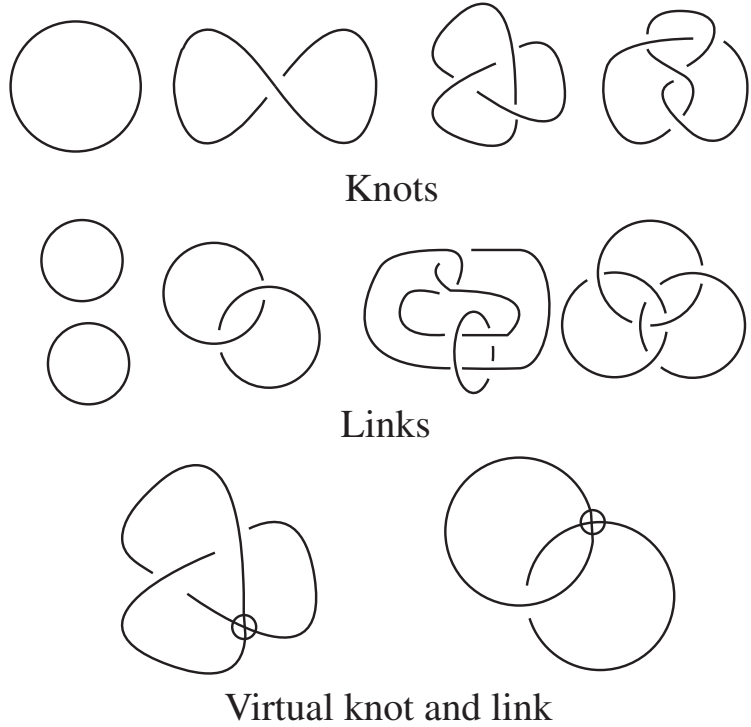}
  \caption{The simplest knots}\label{knlivir}
 \end{figure}

 \begin{definition}
A {\em classical} ({\em virtual}) {\em link} is an equivalence class of classical (virtual) diagrams modulo planar isotopies (diffeomorphisms of the plane on itself preserving the orientation of the plane) and Reidemeister moves (generalised Reidemeister moves). The generalised Reidemeister moves consist of usual Reidemeister moves referring to classical crossings, see Fig.~\ref{rms}, and the {\em detour move} that replaces one arc containing only virtual intersections and self-intersections by another arc of such sort in any other place of the plane, see
Fig.~\ref{detour}.

An {\em oriented classical} ({\em virtual}) {\em link} is an equivalence class of oriented classical (virtual) diagrams modulo planar isotopies (diffeomorphisms of the plane on itself preserving the orientation of the plane) and oriented Reidemeister moves (oriented generalised Reidemeister moves), see Fig.~\ref{rdmst} for oriented Reidemeister moves. The oriented generalised Reidemeister moves consist of usual oriented Reidemeister moves and the detour move.
 \end{definition}

 \begin{remark}
It is worth saying that classical knot theory embeds in virtual knot theory~\cite{Mant35,Mant41} and this fact is not trivial.
 \end{remark}

 \begin{figure}
  \centering\includegraphics[width=300pt]{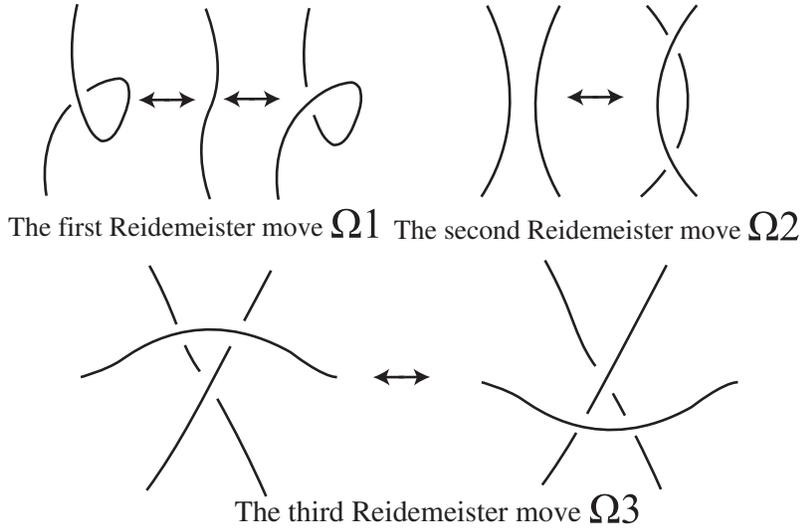}
  \caption{Reidemeister moves $\Omega{1},\,\Omega{2},\,\Omega{3}$}\label{rms}
 \end{figure}

 \begin{figure}
  \centering\includegraphics[width=250pt]{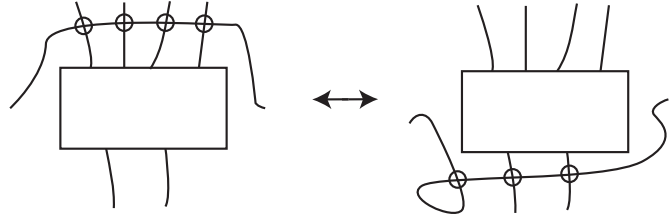}
  \caption{Detour move}\label{detour}
 \end{figure}

 \subsection{Biquandle}

 \begin{definition}\label{def:biq}
A {\em biquandle}~\cite{FJK,HrKa2,KaeKa2,KM1,NOR,NR} is a set $X$ with two binary operations $\circ,\ast\colon X\times X\to X$ satisfying the following axioms:
  \begin{enumerate}
   \item[(R1)]
$x\circ x=x\ast x$ for $\forall\, x\in X$,
   \item[(R2)]
for any $y\in X$ the maps $\alpha_y,\,\beta_y\colon X\to X$ defined by $\alpha_y(x)=x\ast y$, $\beta_y(x)=x\circ y$ are invertible, i.e., for any $z_1,\,z_2\in X$ there exist $x_1,\,x_2\in X$ such that $x_1\ast y=z_1$, $x_2\circ y=z_2$,
   \item[(R3)]
the map $S\colon X\times X\to X\times X$ defined by $S(x,y)=(y\ast x,x\circ y)$ is invertible, i.e., for any $(z,w)\in X\times X$ there exists $(x,y)\in X\times X$ such that $(y\ast x,x\circ y)=(z,w)$,
   \item[(R4)]
the {\em exchange laws} holds\/:
  \begin{gather*}
(x\circ z)\circ(y\circ z)=(x\circ y)\circ(z\ast y),\\
(y\circ z)\ast(x\circ z)=(y\ast x)\circ(z\ast x),\\
(z\ast x)\ast(y\ast x)=(z\ast y)\ast(x\circ y).
  \end{gather*}
  \end{enumerate}

If $X$ and $Y$ are biquandles then a {\em biquandle homomorphism} is a map $f\colon X\to Y$ such that $f(x*y)=f(x)*f(y)$ and $f(x\circ y)=f(x)\circ f(y)$ for $\forall\,x,\,y\in X$.
 \end{definition}

 \begin{definition}
Let $X$ be a finite biquandle and $L$ be an oriented link diagram with $n$ crossings.

The {\em fundamental biquandle} $\mathcal{B}(L)$ of $L$ is the set with biquandle operations consisting of equivalence classes of words in a set of generators corresponding to the edges of $L$ modulo the equivalence relation generated by the crossing relations
 \begin{center}
\biqcol
  \end{center}
of $L$ and the biquandle axioms:
 $$
\mathcal{B}(L)=\langle x_1,\dots,x_{2n}\,|\,\mathrm{CR},\,\mathrm{R1},\,\mathrm{R2},\,\mathrm{R3},\,\mathrm{R4}\rangle.
 $$

A {\em biquandle colouring} (or an {\em $X$-colouring}\/) of $L$ is an assignment of elements of $X$ to the edges in $L$ such that the crossing relations are satisfied at every crossing, i.e., it is a biquandle homomorphism $f\colon\mathcal{B}(L)\to X$. The set of biquandle colourings of $L$ is identified with the set $\mathrm{Hom}(\mathcal{B}(L),X)$ of biquandle homomorphisms from the fundamental biquandle of $L$ to $X$.
 \end{definition}

 \begin{remark}
The biquandle axioms are the conditions required for every biquandle colouring of the edges in a knot diagram before a move and after the move, see Fig.~\ref{rdmst}. Therefore, the number of biquandle colourings is an invariant.
 \end{remark}

 \begin{figure}
  \centering\includegraphics[width=300pt]{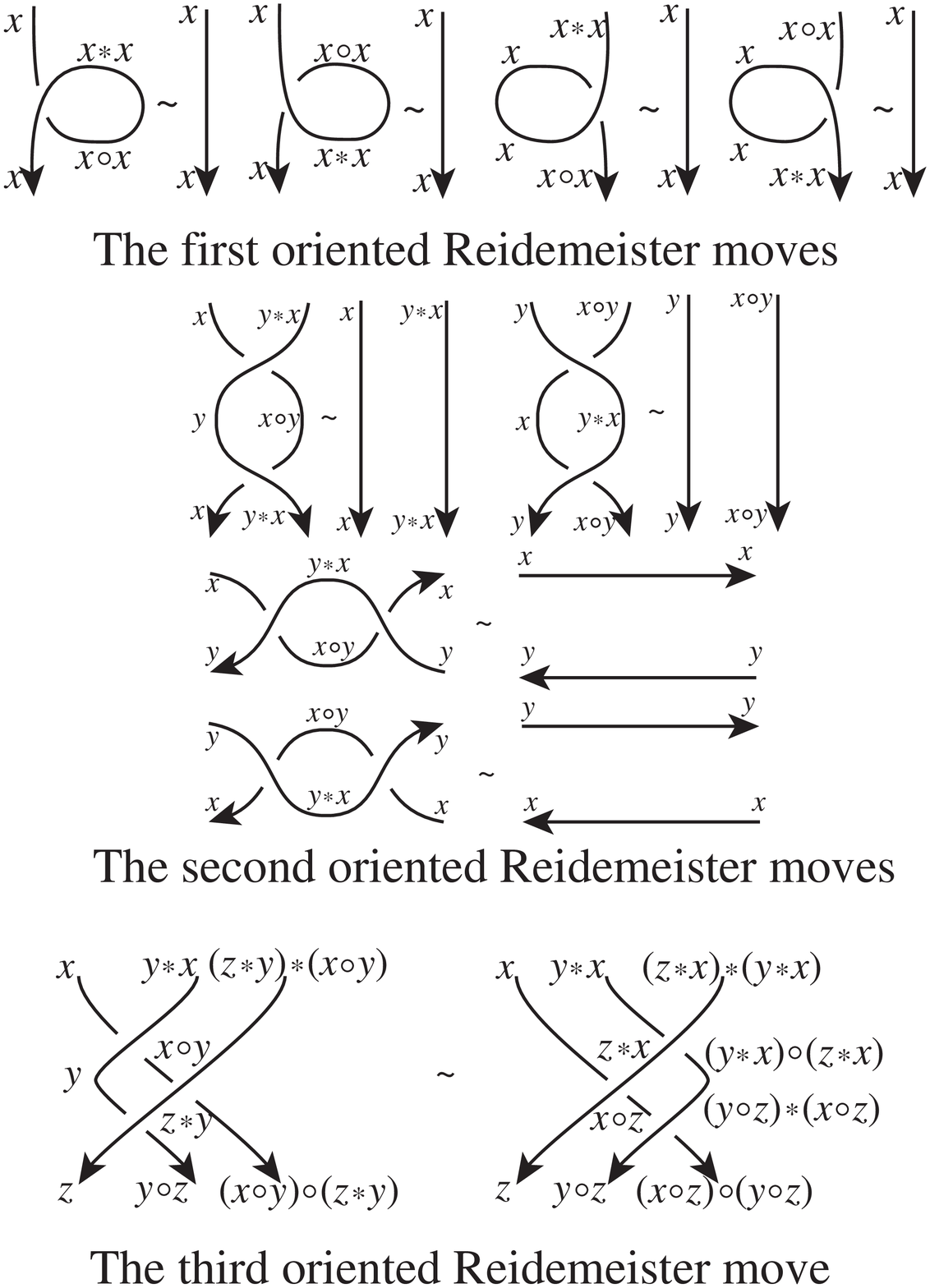}
  \caption{The oriented Reidemeister moves}\label{rdmst}
 \end{figure}

\section{The parity-biquandle bracket and minimality}\label{sec:parbiqbr}

In this section we use the construction of the parity-biquandle bracket from~\cite{IM_bq}.

Let $X$ be a finite biquandle and  $L$ be an oriented virtual link diagram. Let us consider the following set of skein relations:
 \begin{center}
\biqbrdef
  \end{center}
(virtual crossings are disregarded). Here $A_{x,y}$, $B_{x,y}$, $C_{x,y}$, $D_{x,y}$, $E_{x,y}$ and $F_{x,y}$, $x,\,y\in X$, are
variables.

Our purpose is to write down a whole set of relations such that the skein relations give rise to a picture-valued link invariant.

Let $\delta$ be a formal variable.

Let us consider two sets: $\mathfrak{G}_{1,\delta}$ is the set of all equivalence classes of $4$-valent graphs with a cross structure
modulo the second Reidemeister move, see Fig.~\ref{rdmst2fn}, and $L\sqcup \bigcirc=\delta L$, and $\mathfrak{G}_{2,\delta}$ is the
set of all equivalence classes of $4$-valent graphs with a cross structure modulo the first Reidemeister move, see Fig.~\ref{rdmst1fn}, the second Reidemeister move and $L\sqcup \bigcirc=\delta L$.

 \begin{figure}
  \centering\includegraphics[width=170pt]{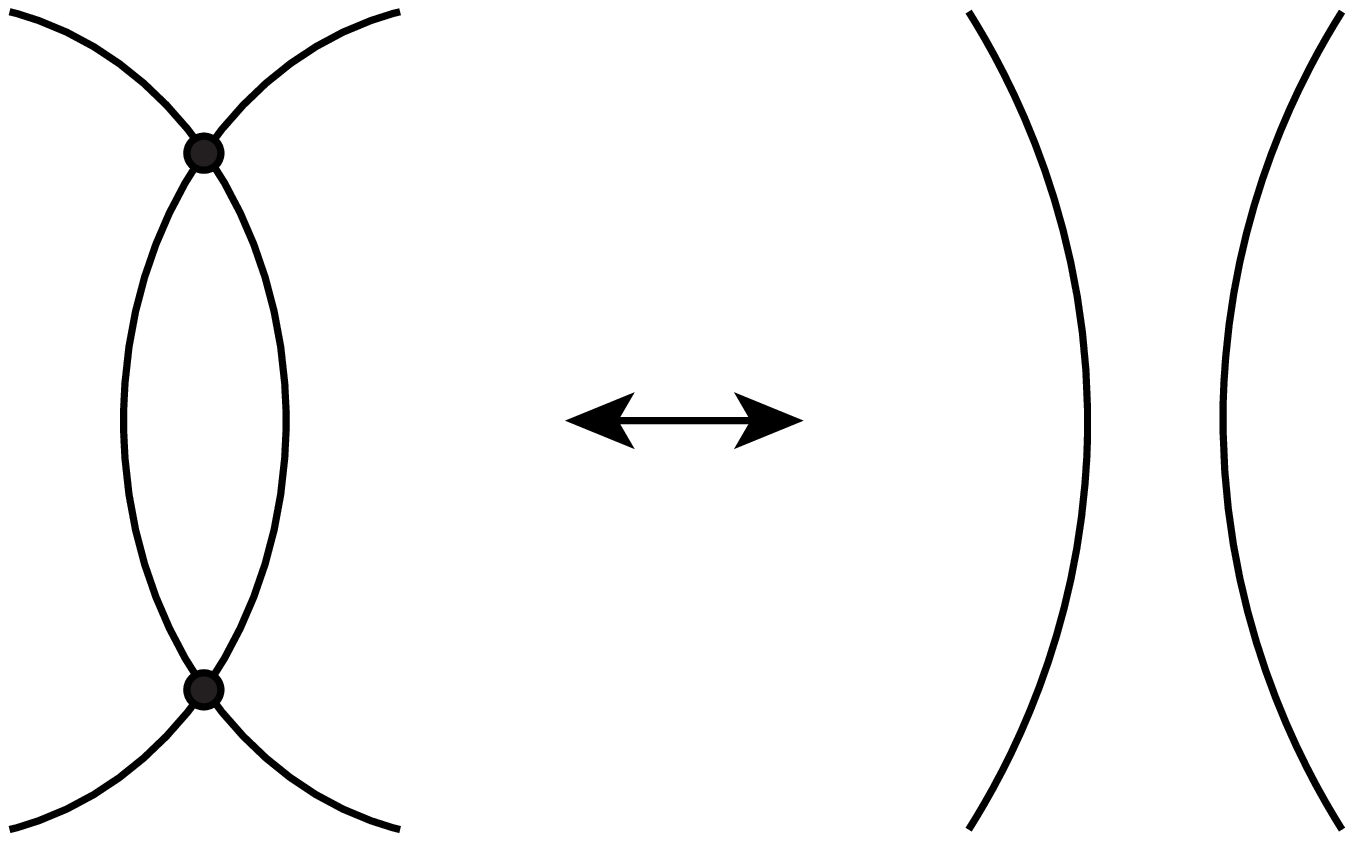}
  \caption{Second Reidemeister move}\label{rdmst2fn}
 \end{figure}

  \begin{figure}
  \centering\includegraphics[width=100pt]{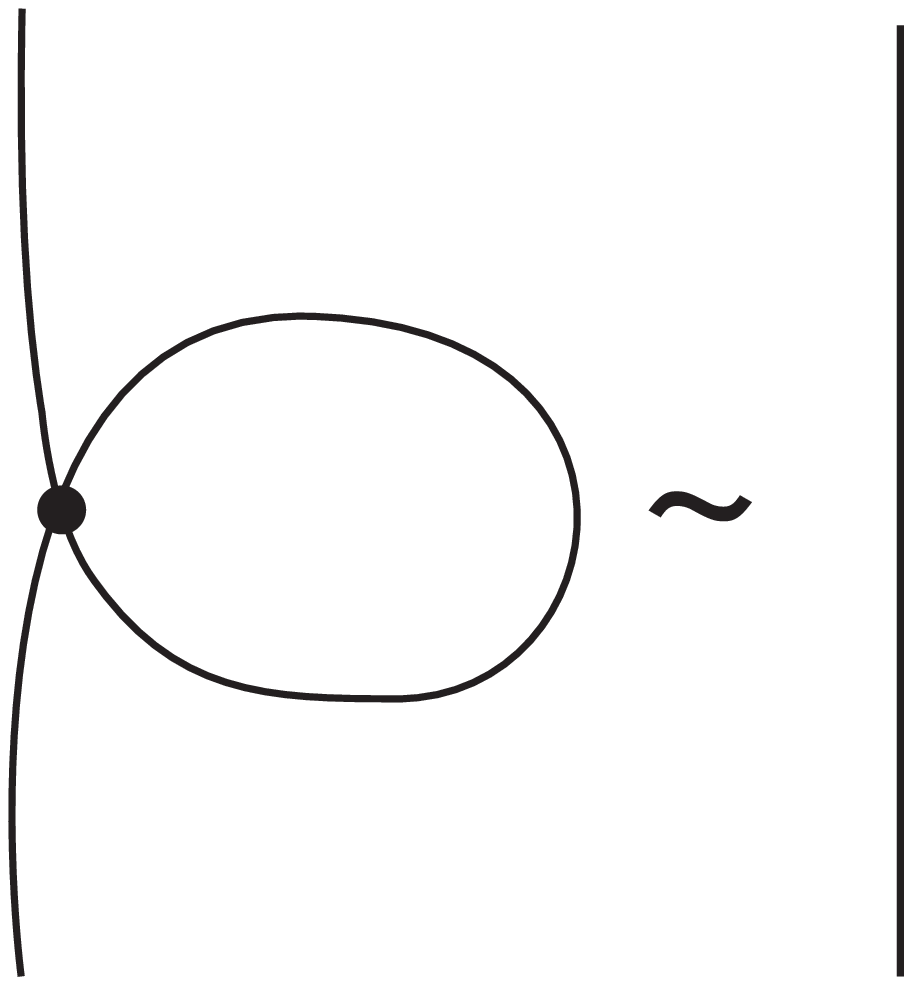}
  \caption{First Reidemeister move}\label{rdmst1fn}
 \end{figure}

The values of the new brackets lie in the module $R\mathfrak{G}_{i,\delta}$, $i=1,2$, where $R$ is a quotient ring of the ring $\mathbb{Z}[A_{x,y},B_{x,y},C_{x,y},D_{x,y},E_{x,y},F_{x,y}\,|\,x,y\in X]$ of polynomials at the variables $A_{x,y}$, $B_{x,y}$, $C_{x,y}$, $D_{x,y}$, $E_{x,y}$, $F_{x,y}$, $x,\,y\in X$, modulo some relations. To get these relations let us apply skein relations to the oriented Reidemeister moves, see Fig.~\ref{rdmst}.

The first Reidemeister moves, see Fig.~\ref{bbr1}, give us the following relations:
 $$
\delta A_{x,x}+B_{x,x}+C_{x,x}=1,\quad \delta
D_{x,x}+E_{x,x}+F_{x,x}=1
 $$
for $\forall\,x\in X$ in the case of $\mathfrak{G}_{i,\delta}$, $i=1,2$, and additional relations
 $$
C_{x,x}=F_{x,x}=0
 $$
for $\forall\,x\in X$ in the case of $\mathfrak{G}_{1,\delta}$.

 \begin{figure}
  \centering\includegraphics[width=400pt]{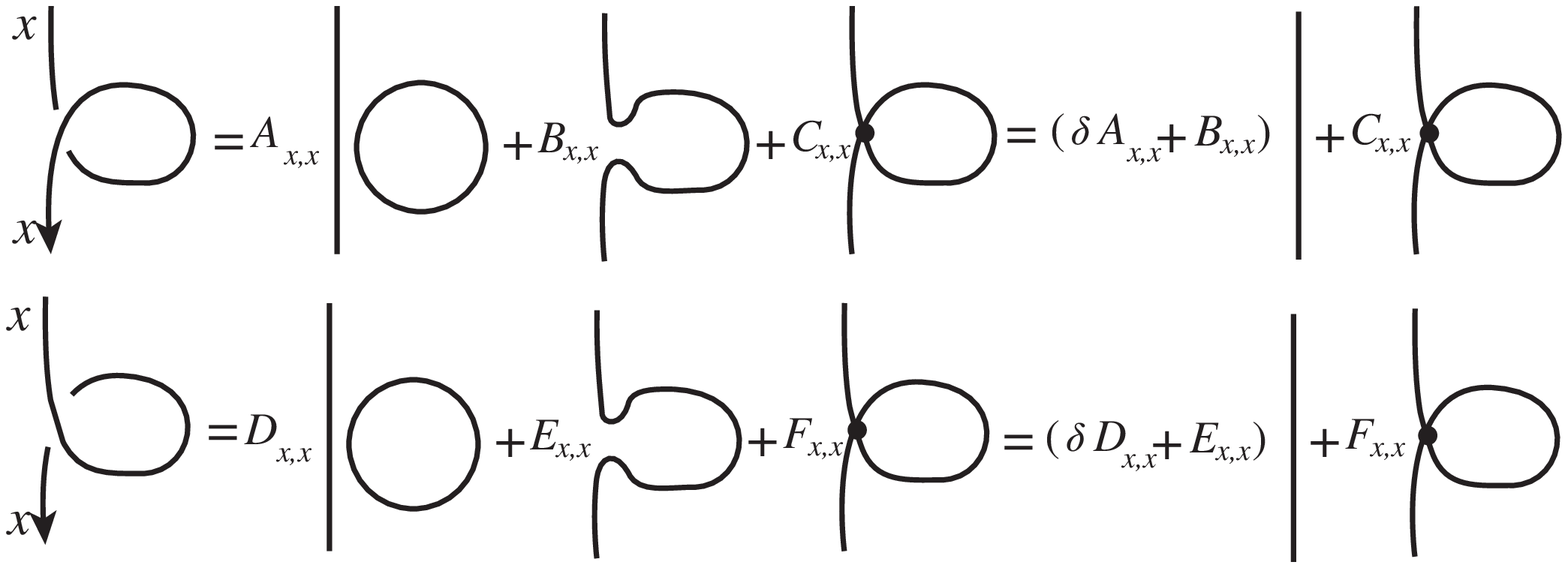}
  \caption{Relations from $\Omega{1}$}\label{bbr1}
 \end{figure}

The first two versions of second Reidemeister moves where the strands are oriented in the same direction, see Fig.~\ref{bbr21},
give us the following relations:
 \begin{gather*}
A_{x,y}D_{x,y}+C_{x,y}F_{x,y}=1,\quad
A_{x,y}F_{x,y}+C_{x,y}D_{x,y}=0,\\
B_{x,y}F_{x,y}+C_{x,y}E_{x,y}+A_{x,y}E_{x,y}+B_{x,y}D_{x,y}+\delta
B_{x,y}E_{x,y}=0
 \end{gather*}
for $\forall\,x,\,y\in X$ in the case of $\mathfrak{G}_{i,\delta}$, $i=1,2$, and additional relations
 \begin{gather*}
B_{x,y}F_{x,y}=C_{x,y}E_{x,y}=0,
 \end{gather*}
for $\forall\,x,\,y\in X$ in the case of $\mathfrak{G}_{1,\delta}$.

 \begin{figure}
  \centering\includegraphics[width=400pt]{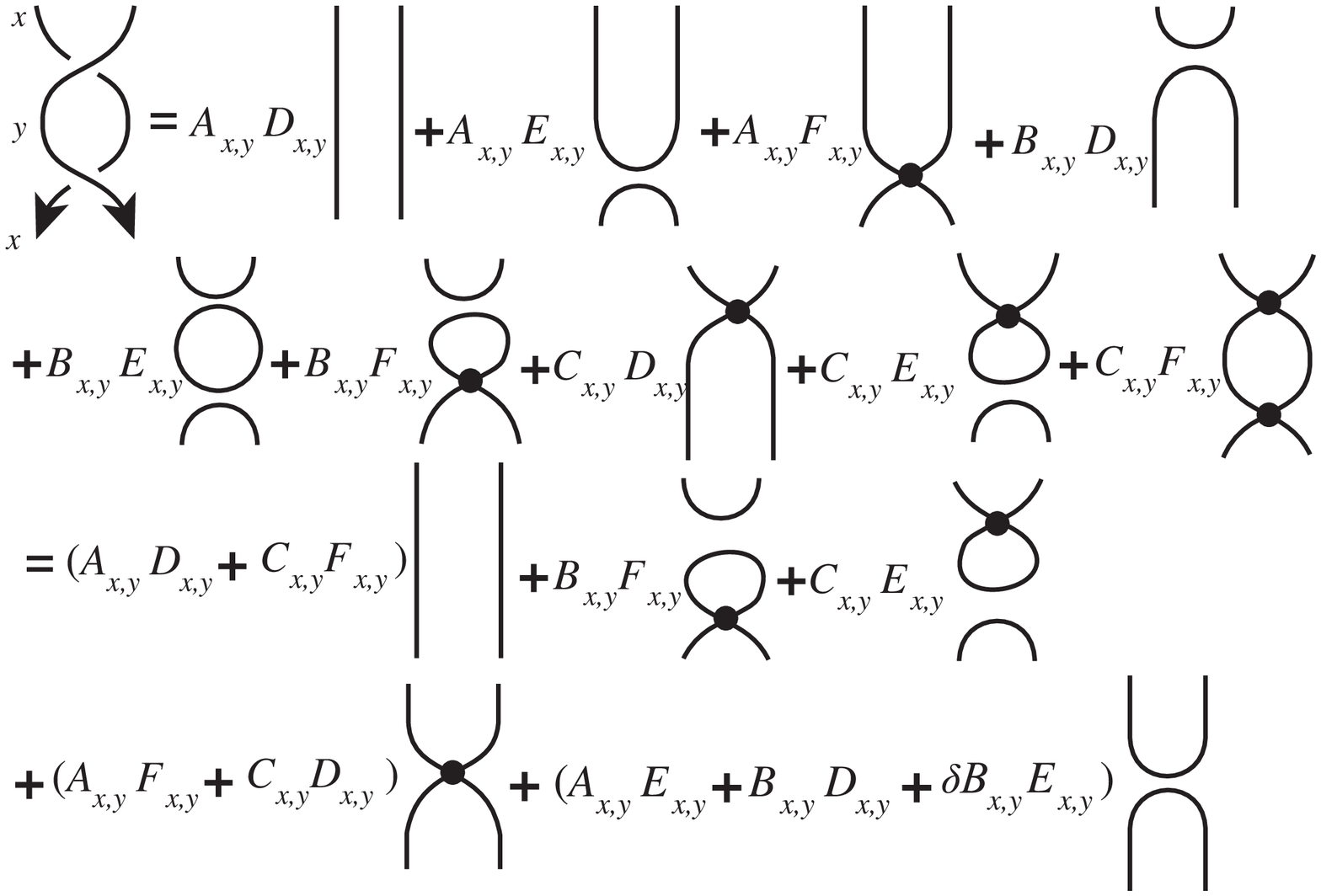}
  \caption{Relations from $\Omega{2}$}\label{bbr21}
 \end{figure}

The last two versions of second Reidemeister moves where the strands are oriented in opposite directions, see Fig.~\ref{bbr22}, give us the following additional relations:
 \begin{gather*}
B_{x,y}E_{x,y}+C_{x,y}F_{x,y}=1,\quad
B_{x,y}F_{x,y}+C_{x,y}E_{x,y}=0,\\
A_{x,y}F_{x,y}+C_{x,y}D_{x,y}+A_{x,y}E_{x,y}+B_{x,y}D_{x,y}+\delta A_{x,y}D_{x,y}=0
 \end{gather*}
for $\forall\,x,\,y\in X$ in the case of $\mathfrak{G}_{i,\delta}$, $i=1,2$, and additional relations
 \begin{gather*}
A_{x,y}F_{x,y}=C_{x,y}D_{x,y}=0
 \end{gather*}
for $\forall\,x,\,y\in X$ in the case of $\mathfrak{G}_{1,\delta}$.

 \begin{figure}
  \centering\includegraphics[width=400pt]{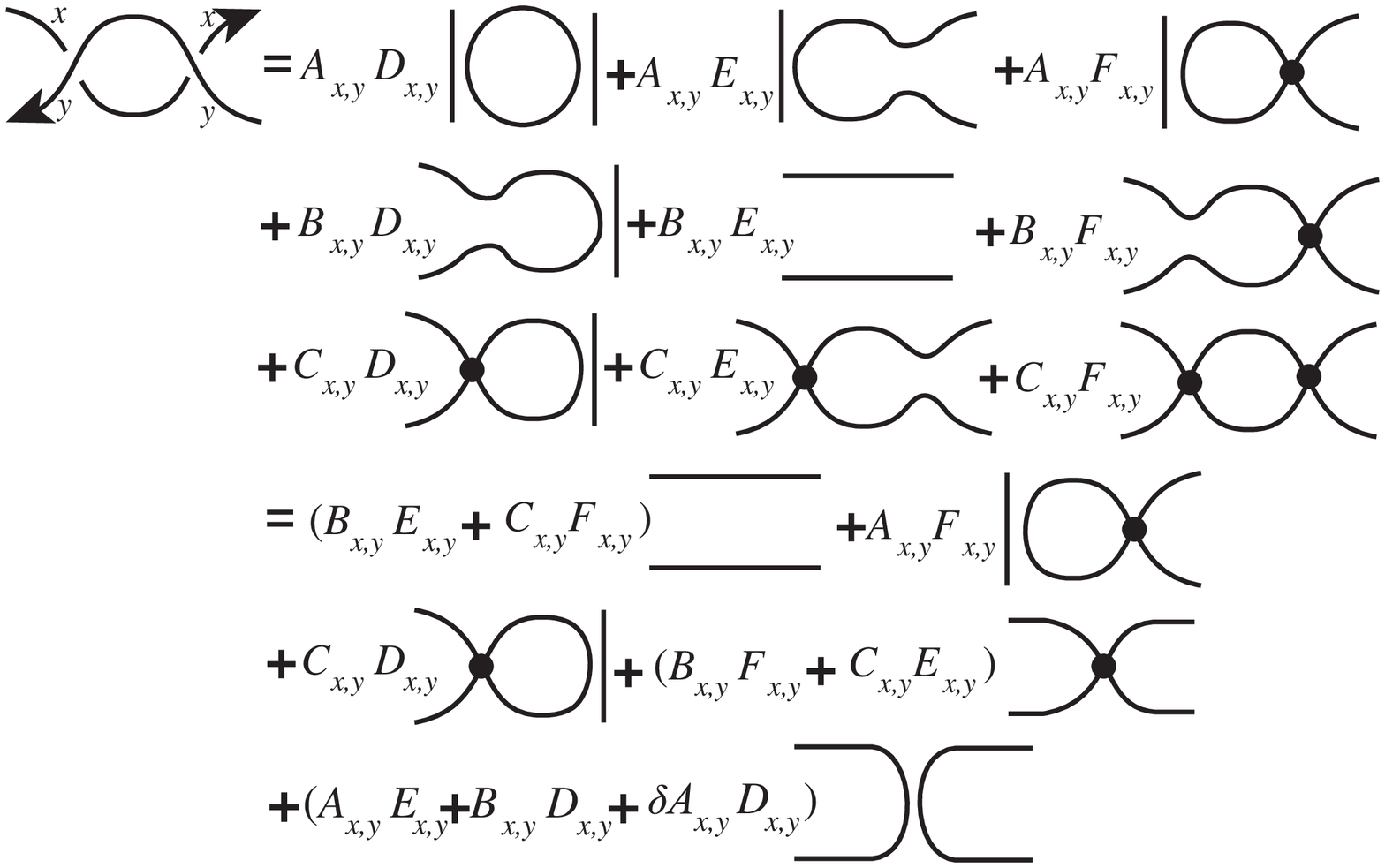}
  \caption{Relations from $\Omega{2}$}\label{bbr22}
 \end{figure}

The third Reidemeister moves, see Fig.~\ref{bbr31},~\ref{bbr32}, give us the following relations:
 \begin{gather*}
A_{x,y}A_{y,z}A_{x\circ y,z\ast y}+C_{x,y}C_{y,z}A_{x\circ y,z\ast y}=A_{x,z}A_{y\ast x,z\ast x}A_{x\circ z,y\circ z}+A_{x,z}C_{y\ast x,z\ast x}C_{x\circ z,y\circ z},\\
A_{x,y}B_{y,z}B_{x\circ y,z\ast y}+C_{x,y}B_{y,z}C_{x\circ y,z\ast y}=B_{x,z}B_{y\ast x,z\ast x}A_{x\circ z,y\circ z}+C_{x,z}B_{y\ast x,z\ast x}C_{x\circ z,y\circ z},\\
B_{x,y}A_{y,z}B_{x\circ y,z\ast y}+B_{x,y}C_{y,z}C_{x\circ y,z\ast y}=B_{x,z}A_{y\ast x,z\ast x}B_{x\circ z,y\circ z}+C_{x,z}C_{y\ast x,z\ast x}B_{x\circ z,y\circ z},\\
A_{x,y}C_{y,z}A_{x\circ y,z\ast y}+C_{x,y}A_{y,z}A_{x\circ y,z\ast y}=C_{x,z}A_{y\ast x,z\ast x}A_{x\circ z,y\circ z},\\
A_{x,y}A_{y,z}C_{x\circ y,z\ast y}=A_{x,z}A_{y\ast x,z\ast x}C_{x\circ z,y\circ z}+A_{x,z}C_{y\ast x,z\ast x}A_{x\circ z,y\circ z},\\
A_{x,y}C_{y,z}B_{x\circ y,z\ast y}=B_{x,z}B_{y\ast x,z\ast x}C_{x\circ z,y\circ z}+C_{x,z}B_{y\ast x,z\ast x}A_{x\circ z,y\circ z},\\
B_{x,y}C_{y,z}B_{x\circ y,z\ast y}+B_{x,z}A_{y,z}C_{x\circ y,z\ast y}=B_{x,z}A_{y\ast x,z\ast x}C_{x\circ z,y\circ z},\\
A_{x,y}B_{y,z}C_{x\circ y,z\ast y}+C_{x,z}B_{y,z}B_{x\circ y,z\ast y}=B_{x,z}C_{y\ast x,z\ast x}A_{x\circ z,y\circ z},\\
C_{x,y}A_{y,z}B_{x\circ y,z\ast y}=B_{x,z}C_{y\ast x,z\ast x}B_{x\circ z,y\circ z}+C_{x,z}A_{y\ast x,z\ast x}B_{x\circ z,y\circ z},\\
C_{x,y}C_{y,z}B_{x\circ y,z\ast y}=B_{x,z}C_{y\ast x,z\ast x}C_{x\circ z,y\circ z},\\
A_{x,y}C_{y,z}C_{x\circ y,z\ast y}=C_{x,z}C_{y\ast x,z\ast x}A_{x\circ z,y\circ z},\\
C_{x,y}A_{y,z}C_{x\circ y,z\ast y}=C_{x,z}A_{y\ast x,z\ast x}C_{x\circ z,y\circ z},\\
C_{x,y}C_{y,z}C_{x\circ y,z\ast y}=C_{x,z}C_{y\ast x,z\ast x}C_{x\circ z,y\circ z}=0,
 \end{gather*}
 \begin{gather*}
A_{x,y}A_{y,z}B_{x\circ y,z\ast y}=A_{x,z}B_{y\ast x,z\ast x}A_{x\circ z,y\circ z}+
A_{x,z}A_{y\ast x,z\ast x}B_{x\circ z,y\circ z}\\
+\delta A_{x,z}B_{y\ast x,z\ast x}B_{x\circ z,y\circ
z}+B_{x,z}B_{y\ast x,z\ast x}B_{x\circ z,y\circ z}+ A_{x,z}B_{y\ast
x,z\ast x}C_{x\circ z,y\circ z}\\
+A_{x,z}C_{y\ast x,z\ast x}B_{x\circ z,y\circ z}+C_{x,z}B_{y\ast x,z\ast x}B_{x\circ z,y\circ z},\\
B_{x,z}A_{y\ast x,z\ast x}A_{x\circ z,y\circ z}=B_{x,y}A_{y,z}A_{x\circ y,z\ast y}+A_{x,y}B_{y,z}A_{x\circ y,z\ast y}\\
+\delta B_{x,y}B_{y,z}A_{x\circ y,z\ast y}+B_{x,y}B_{y,z}B_{x\circ
y,z\ast y}+B_{x,y}C_{y,z}A_{x\circ y,z\ast
y}+C_{x,y}B_{y,z}A_{x\circ y,z\ast y}+B_{x,y}B_{y,z}C_{x\circ
y,z\ast y}
   \end{gather*}
for $\forall\,x,\,y,\,z\in X$ in the case of $\mathfrak{G}_{i,\delta}$, $i=1,2$, and additional relations
 \begin{gather*}
B_{x,y}C_{y,z}A_{x\circ y,z\ast y}=C_{x,y}B_{y,z}A_{x\circ y,z\ast y}=B_{x,y}B_{y,z}C_{x\circ y,z\ast y}=0,\\
A_{x,z}B_{y\ast x,z\ast x}C_{x\circ z,y\circ z}=A_{x,z}C_{y\ast x,z\ast x}B_{x\circ z,y\circ z}=C_{x,z}B_{y\ast x,z\ast x}B_{x\circ z,y\circ z}=0
   \end{gather*}
for $\forall\,x,\,y,\,z\in X$ in the case of $\mathfrak{G}_{1,\delta}$.

 \begin{figure}
  \centering\includegraphics[width=420pt]{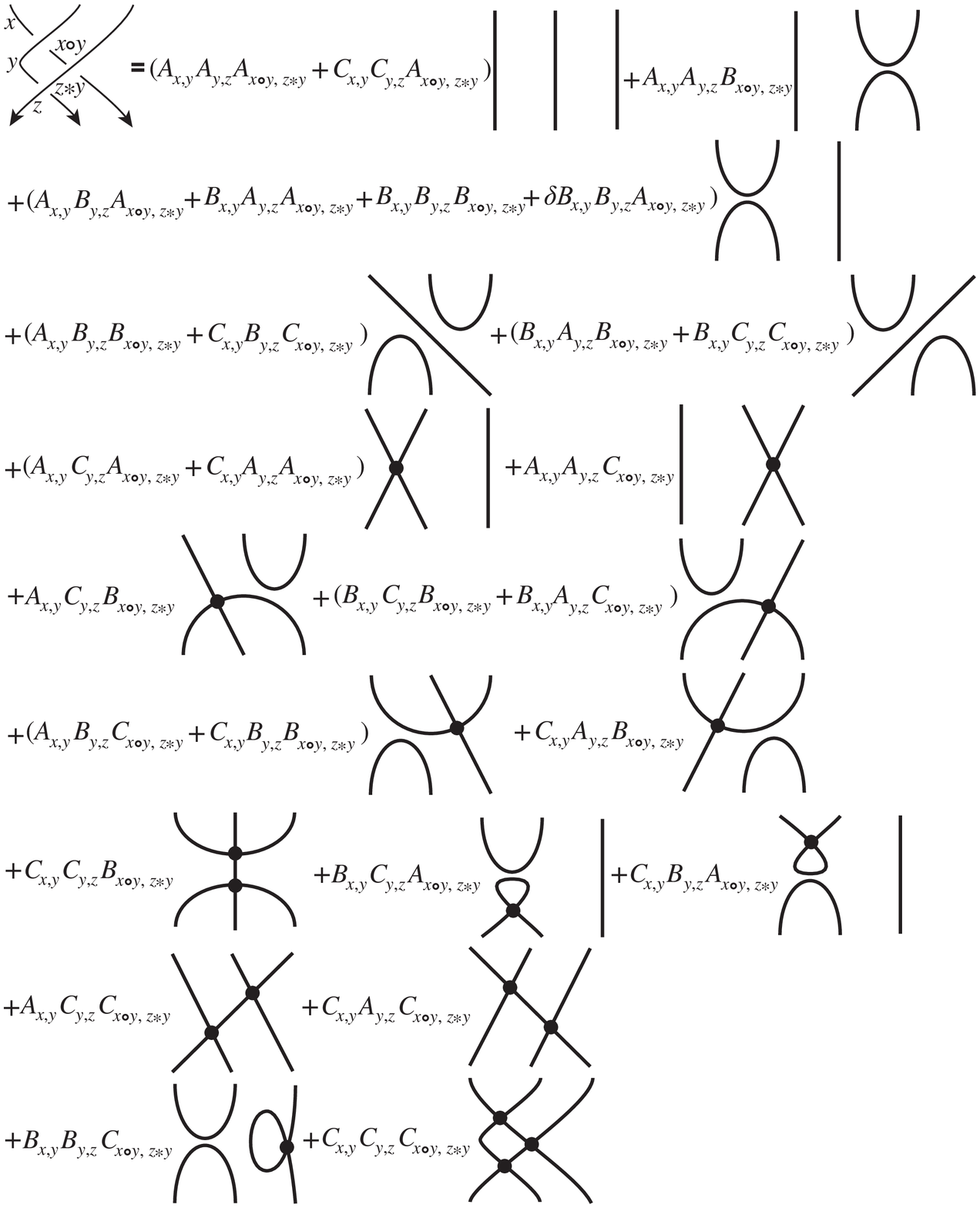}
  \caption{Relations from $\Omega{3}$}\label{bbr31}
 \end{figure}

 \begin{figure}
  \centering\includegraphics[width=440pt]{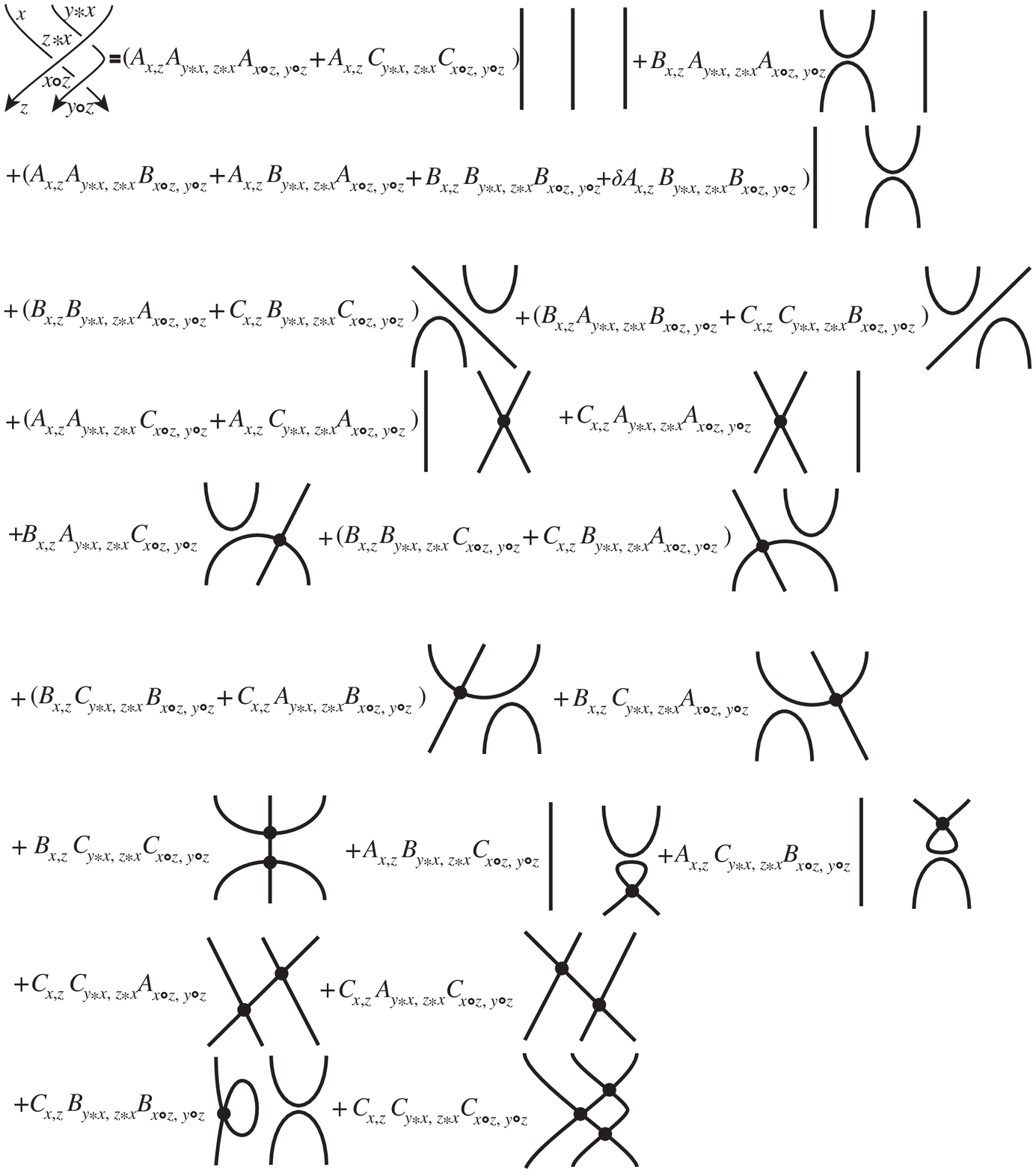}
  \caption{Relations from $\Omega{3}$}\label{bbr32}
 \end{figure}

As a result, we get the following definitions (cf.~\cite{NOR}).

Let $L$ be an oriented (virtual) link diagram with $n$ crossings and let
 $$
\mathcal{B}(L)=\langle x_1,\dots,x_{2n}\,|\,\mathrm{CR},\,\mathrm{R1},\,\mathrm{R2},\,\mathrm{R3},\,\mathrm{R4}\rangle.
 $$
be its fundamental biquandle. There are $3^n$ states of $L$, i.e., at each crossing we have either the positive smoothing, or negative
smoothing, or the graphical vertex  (we disregard virtual crossings). For each state we have the contribution: the product of $n$ variables of $A_{x,y}$, or $B_{x,y}$, or $C_{x,y}$, or $D_{x,y}$, or $E_{x,y}$, or $F_{x,y}$ times $\delta^k$, where $k$ or $k+1$ is the number of circles in the state ($k+1$ takes place in the case when the smoothing consists of $k+1$ circles), and a $4$-valent graph with a cross structure (this $4$-valent graph can be a circle but not a disjoint union of a circle and a $4$-valent graph with a cross structure or circles).

 \begin{definition}
The {\em fundamental parity-biquandle bracket value} $[[L]]$ for $L$ is the sum of the contributions.
 \end{definition}

Let $X$ be a finite biquandle and $f\in\mathrm{Hom}(\mathcal{B}(L),X)$ an $X$-colouring of $L$. In the polynomial ring $\mathbb{Z}[A_{x,y},B_{x,y},C_{x,y},D_{x,y},E_{x,y},F_{x,y}\,|\,x,y\in X]$ we consider the ideals $I_j$, $j=1,2$, generated by the
following polynomials:
 \begin{enumerate}
  \item[$(i)_1$]
for $\forall\,x\in X$,
 $$
\delta A_{x,x}+B_{x,x}-1,\qquad \delta D_{x,x}+E_{x,x}-1\qquad\hbox{and}\qquad C_{x,x},\,F_{x,x},
 $$
  \item[$(ii)_1$]
$\forall\,x,y\in X$,
 \begin{gather*}
A_{x,y}F_{x,y},\qquad C_{x,y}D_{x,y},\qquad B_{x,y}F_{x,y},\qquad C_{x,y}E_{x,y},\qquad A_{x,y}D_{x,y}-B_{x,y}E_{x,y},\\ A_{x,y}D_{x,y}+C_{x,y}F_{x,y}-1\qquad\hbox{and}\qquad \delta A_{x,y}D_{x,y}+A_{x,y}E_{x,y}+B_{x,y}D_{x,y},
  \end{gather*}
  \item[$(iii)_1$]
for $\forall\,x,y,z\in X$,
 \begin{gather*}
A_{x,y}A_{y,z}A_{x\circ y,z\ast y}+C_{x,y}C_{y,z}A_{x\circ y,z\ast y}-(A_{x,z}A_{y\ast x,z\ast x}A_{x\circ z,y\circ z}+A_{x,z}C_{y\ast x,z\ast x}C_{x\circ z,y\circ z}),\\
A_{x,y}B_{y,z}B_{x\circ y,z\ast y}+C_{x,y}B_{y,z}C_{x\circ y,z\ast y}-(B_{x,z}B_{y\ast x,z\ast x}A_{x\circ z,y\circ z}+C_{x,z}B_{y\ast x,z\ast x}C_{x\circ z,y\circ z}),\\
B_{x,y}A_{y,z}B_{x\circ y,z\ast y}+B_{x,y}C_{y,z}C_{x\circ y,z\ast y}-(B_{x,z}A_{y\ast x,z\ast x}B_{x\circ z,y\circ z}+C_{x,z}C_{y\ast x,z\ast x}B_{x\circ z,y\circ z}),\\
A_{x,y}C_{y,z}A_{x\circ y,z\ast y}+C_{x,y}A_{y,z}A_{x\circ y,z\ast y}-C_{x,z}A_{y\ast x,z\ast x}A_{x\circ z,y\circ z},\\
A_{x,y}A_{y,z}C_{x\circ y,z\ast y}-(A_{x,z}A_{y\ast x,z\ast x}C_{x\circ z,y\circ z}+A_{x,z}C_{y\ast x,z\ast x}A_{x\circ z,y\circ z}),\\
A_{x,y}C_{y,z}B_{x\circ y,z\ast y}-(B_{x,z}B_{y\ast x,z\ast x}C_{x\circ z,y\circ z}+C_{x,z}B_{y\ast x,z\ast x}A_{x\circ z,y\circ z}),\\
B_{x,y}C_{y,z}B_{x\circ y,z\ast y}+B_{x,z}A_{y,z}C_{x\circ y,z\ast y}-B_{x,z}A_{y\ast x,z\ast x}C_{x\circ z,y\circ z},\\
A_{x,y}B_{y,z}C_{x\circ y,z\ast y}+C_{x,z}B_{y,z}B_{x\circ y,z\ast y}-B_{x,z}C_{y\ast x,z\ast x}A_{x\circ z,y\circ z},\\
C_{x,y}A_{y,z}B_{x\circ y,z\ast y}-(B_{x,z}C_{y\ast x,z\ast x}B_{x\circ z,y\circ z}+C_{x,z}A_{y\ast x,z\ast x}B_{x\circ z,y\circ z}),\\
C_{x,y}C_{y,z}B_{x\circ y,z\ast y}-B_{x,z}C_{y\ast x,z\ast x}C_{x\circ z,y\circ z},\\
A_{x,y}C_{y,z}C_{x\circ y,z\ast y}-C_{x,z}C_{y\ast x,z\ast x}A_{x\circ z,y\circ z},\\
C_{x,y}A_{y,z}C_{x\circ y,z\ast y}-C_{x,z}A_{y\ast x,z\ast x}C_{x\circ z,y\circ z},\\
B_{x,y}C_{y,z}A_{x\circ y,z\ast y},\,C_{x,y}B_{y,z}A_{x\circ y,z\ast y},\,B_{x,y}B_{y,z}C_{x\circ y,z\ast y},\
C_{x,y}C_{y,z}C_{x\circ y,z\ast y},
 \end{gather*}
 \begin{gather*}
A_{x,z}B_{y\ast x,z\ast x}C_{x\circ z,y\circ z},\,A_{x,z}C_{y\ast x,z\ast x}B_{x\circ z,y\circ z},\\
C_{x,z}B_{y\ast x,z\ast x}B_{x\circ z,y\circ z},\,C_{x,z}C_{y\ast x,z\ast x}C_{x\circ z,y\circ z},\\
A_{x,y}A_{y,z}B_{x\circ y,z\ast y}-(A_{x,z}B_{y\ast x,z\ast x}A_{x\circ z,y\circ z}+A_{x,z}A_{y\ast x,z\ast x}B_{x\circ z,y\circ z}\\
+\delta A_{x,z}B_{y\ast x,z\ast x}B_{x\circ z,y\circ z}+B_{x,z}B_{y\ast x,z\ast x}B_{x\circ z,y\circ z}),\\
B_{x,z}A_{y\ast x,z\ast x}A_{x\circ z,y\circ z}-(B_{x,y}A_{y,z}A_{x\circ y,z\ast y}+A_{x,y}B_{y,z}A_{x\circ y,z\ast y}\\
+\delta B_{x,y}B_{y,z}A_{x\circ y,z\ast y}+B_{x,y}B_{y,z}B_{x\circ y,z\ast y})
   \end{gather*}
 \end{enumerate}
in the case of $\mathfrak{G}_{1,\delta}$ and
 \begin{enumerate}
  \item[$(i)_2$]
for $\forall\,x\in X$,
 $$
\delta A_{x,x}+B_{x,x}+C_{x,x}-1\qquad\hbox{and}\qquad\delta D_{x,x}+E_{x,x}+F_{x,x}-1,
 $$
  \item[$(ii)_2$]
for $\forall\,x,y\in X$,
 \begin{gather*}
A_{x,y}F_{x,y}+C_{x,y}D_{x,y},\qquad B_{x,y}F_{x,y}+C_{x,y}E_{x,y}\qquad A_{x,y}D_{x,y}-B_{x,y}E_{x,y},\\
A_{x,y}D_{x,y}+C_{x,y}F_{x,y}-1\qquad\hbox{and}\qquad\delta A_{x,y}D_{x,y}+A_{x,y}E_{x,y}+B_{x,y}D_{x,y},
 \end{gather*}
  \item[$(iii)_2$]
for $\forall\,x,y,z\in X$,
 \begin{gather*}
A_{x,y}A_{y,z}A_{x\circ y,z\ast y}+C_{x,y}C_{y,z}A_{x\circ y,z\ast y}-(A_{x,z}A_{y\ast x,z\ast x}A_{x\circ z,y\circ z}+A_{x,z}C_{y\ast x,z\ast x}C_{x\circ z,y\circ z}),\\
A_{x,y}B_{y,z}B_{x\circ y,z\ast y}+C_{x,y}B_{y,z}C_{x\circ y,z\ast y}-(B_{x,z}B_{y\ast x,z\ast x}A_{x\circ z,y\circ z}+C_{x,z}B_{y\ast x,z\ast x}C_{x\circ z,y\circ z}),\\
B_{x,y}A_{y,z}B_{x\circ y,z\ast y}+B_{x,y}C_{y,z}C_{x\circ y,z\ast y}-(B_{x,z}A_{y\ast x,z\ast x}B_{x\circ z,y\circ z}+C_{x,z}C_{y\ast x,z\ast x}B_{x\circ z,y\circ z}),\\
A_{x,y}C_{y,z}A_{x\circ y,z\ast y}+C_{x,y}A_{y,z}A_{x\circ y,z\ast y}-C_{x,z}A_{y\ast x,z\ast x}A_{x\circ z,y\circ z},\\
A_{x,y}A_{y,z}C_{x\circ y,z\ast y}-(A_{x,z}A_{y\ast x,z\ast x}C_{x\circ z,y\circ z}+A_{x,z}C_{y\ast x,z\ast x}A_{x\circ z,y\circ z}),\\
A_{x,y}C_{y,z}B_{x\circ y,z\ast y}-(B_{x,z}B_{y\ast x,z\ast x}C_{x\circ z,y\circ z}+C_{x,z}B_{y\ast x,z\ast x}A_{x\circ z,y\circ z}),\\
B_{x,y}C_{y,z}B_{x\circ y,z\ast y}+B_{x,z}A_{y,z}C_{x\circ y,z\ast y}-B_{x,z}A_{y\ast x,z\ast x}C_{x\circ z,y\circ z},\\
A_{x,y}B_{y,z}C_{x\circ y,z\ast y}+C_{x,z}B_{y,z}B_{x\circ y,z\ast y}-B_{x,z}C_{y\ast x,z\ast x}A_{x\circ z,y\circ z},\\
C_{x,y}A_{y,z}B_{x\circ y,z\ast y}-(B_{x,z}C_{y\ast x,z\ast x}B_{x\circ z,y\circ z}+C_{x,z}A_{y\ast x,z\ast x}B_{x\circ z,y\circ z}),\\
C_{x,y}C_{y,z}B_{x\circ y,z\ast y}-B_{x,z}C_{y\ast x,z\ast x}C_{x\circ z,y\circ z},\\
A_{x,y}C_{y,z}C_{x\circ y,z\ast y}-C_{x,z}C_{y\ast x,z\ast x}A_{x\circ z,y\circ z},\\
C_{x,y}A_{y,z}C_{x\circ y,z\ast y}-C_{x,z}A_{y\ast x,z\ast x}C_{x\circ z,y\circ z},\\
C_{x,y}C_{y,z}C_{x\circ y,z\ast y},\,C_{x,z}C_{y\ast x,z\ast x}C_{x\circ z,y\circ z},\\
A_{x,y}A_{y,z}B_{x\circ y,z\ast y}-(A_{x,z}B_{y\ast x,z\ast x}A_{x\circ z,y\circ z}+
A_{x,z}A_{y\ast x,z\ast x}B_{x\circ z,y\circ z}\\
+\delta A_{x,z}B_{y\ast x,z\ast x}B_{x\circ z,y\circ z}+B_{x,z}B_{y\ast x,z\ast x}B_{x\circ z,y\circ z}+ A_{x,z}B_{y\ast
x,z\ast x}C_{x\circ z,y\circ z}\\
+A_{x,z}C_{y\ast x,z\ast x}B_{x\circ z,y\circ z}+C_{x,z}B_{y\ast x,z\ast x}B_{x\circ z,y\circ z}),\\
B_{x,z}A_{y\ast x,z\ast x}A_{x\circ z,y\circ z}-(B_{x,y}A_{y,z}A_{x\circ y,z\ast y}+A_{x,y}B_{y,z}A_{x\circ y,z\ast y}\\
+\delta B_{x,y}B_{y,z}A_{x\circ y,z\ast y}+B_{x,y}B_{y,z}B_{x\circ y,z\ast y}+B_{x,y}C_{y,z}A_{x\circ y,z\ast
y}+C_{x,y}B_{y,z}A_{x\circ y,z\ast y}+B_{x,y}B_{y,z}C_{x\circ y,z\ast y})
   \end{gather*}
 \end{enumerate}
in the case of $\mathfrak{G}_{2,\delta}$.

Let $R_j=\mathbb{Z}[A_{x,y},B_{x,y},C_{x,y},D_{x,y},E_{x,y},F_{x,y}\,|\,x,y\in X]/(I_j)$, $j=1,2$, be the quotient ring. We set the value of the fundamental parity-biquandle bracket value for $L$ in $f$ for $I_j$ to be $[[L]](f)_{j}\in R_j\mathfrak{G}_{j,\delta}$, $j=1,2$.

The {\em parity-biquandle bracket multiset} of $L$ for $X$ is the following multiset:
 $$
[[L]](X)_{j}=\{[[L]](f)_{j}\,|f\in\mathrm{Hom}(\mathcal{B}(L),X)\},
 $$
$j=1,2$.

 \begin{theorem}
The parity-biquandle bracket multiset is an invariant of virtual links. Namely{\em,} if two virtual link diagrams $L_1$ and $L_2$
represent the same link then $[[L_1]](X)_{j}=[[L_2]](X)_{j}$ for any biquandle $X$.
 \end{theorem}

 \begin{proof}
Let $L_1$ is obtained from $L_2$ by a Reidemeister move. It follows from Def.~\ref{sec:kn&biq&par}.\ref{def:biq} that the fundamental biquandles
$\mathcal{B}(L_1)$ and $\mathcal{B}(L_2)$ are isomorphic and to each biquandle colouring $f_1\in \mathrm{Hom}(\mathcal{B}(L_1),X)$ of
$L_1$ it is associated the biquandle colouring $f_2\in\mathrm{Hom}(\mathcal{B}(L_2),X)$ of $L_2$. Then by using relations for
$I_j$, $j=1,2$, we get that $[[L_1]](f_1)_{j}=[[L_2]](f_2)_{j}$.
 \end{proof}

 \begin{definition}
A $4$-valent graph with a cross structure is called {\em $1$-irreducible} ({\em $2$-irreducible}) if no decreasing second (and first) Reidemeister move (moves) can be applied to it.
 \end{definition}

The next lemma can be easily prove by an induction on the set of vertices and using the fact that if we can apply two decreasing Reidemeister moves (we consider only first and second Reidemeister moves) to a $4$-valent graph with a cross structure then after applying one of them we can apply the other one.

 \begin{lemma}[Diamond lemma]\label{lem:diam}
A $j$-irreducible graph{\em,} $j=1,2,$ has a minimal number of vertices amongst $4$-valent graphs with a cross structure being equivalent to it in $\mathfrak{G}_{j,\delta}$. Moreover{\em,} each $4$-valent graph with a cross structure has a unique $j$-irreducible graph{\em,} $j=1,2,$ being equivalent to it in $\mathfrak{G}_{j,\delta}$.
 \end{lemma}

By using Lemma~\ref{lem:diam} we can consider $[[L]](f)_{j}$ to be the linear combination of $j$-irreducible graphs for any virtual link $L$ and any $X$-colouring $f$ of $L$.

 \begin{definition}
A {\em leading term} in $[[L]](f)_{j}$ is any $j$-irreducible graph having a non-zero coefficient in $[[L]](f)_{j}$ and having the maximal number of vertices amongst all $j$-irreducible graphs with non-zero coefficients in $[[L]](f)_{j}$.
 \end{definition}

From the construction of the parity-biquandle bracket we get the following theorem.

 \begin{theorem}[cf.~\cite{IMN2,IMN,IMN3,Mant31,Mant32,Mant33,Mant35,Mant39,Mant36,Mant38,Mant40,Mant41,Mant42,ManIly}]
Let $G$ be a leading term of $[[L]](f)_{j}$. Then the number of crossings of any link being equivalent to $L$ is greater than or equal to the number of vertices of $G$. Moreover{\em,} if the number of vertices of $G$ equals the number of crossings of $L$ then $L$ is a minimal virtual link in the sense that any virtual link being equivalent to $L$ contains $L$ after smoothing some its vertices.
 \end{theorem}

\section{Examples}\label{sec:examples}

In this section we consider two examples: a classical knot and a classical link.

1) Let $K$ be the knot $8_{18}$ from the Rolfsen knot table, see Fig.~\ref{818} for its representation as the closure of the braid.

 \begin{figure}
  \centering\includegraphics[width=80pt]{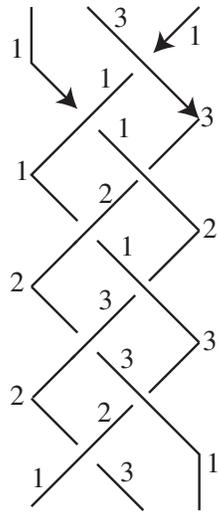}
  \caption{The knot $8_{18}$}\label{818}
 \end{figure}

Let us consider the biquanlde $X=\{1,2,3\}$, where two operations $\circ$ and $\ast$ are defined by the following matrices:
 $$
\circ=\left(\begin{array}{ccc}1&1&1\\3&3&3\\2&2&2\end{array}\right)\qquad\hbox{and}\qquad
\ast=\left(\begin{array}{ccc}1&2&3\\2&3&1\\3&1&2\end{array}\right),
 $$
i.e., $1\circ1=1\circ2=1\circ3=1$, $2\circ1=2\circ2=2\circ3=3$, $3\circ1=3\circ2=3\circ3=2$, $1\ast1=1$, $1\ast2=2$, $1\ast3=3$ and so on. By using this biquandle we colour our diagram by $f$ as shown in Fig.~\ref{818}.

Let us find a leading term of the parity-biquandle bracket $[[K]](f)_2$. It is easy to see that the coefficient before the graph $G_{8_{18}}$ obtained from the knot $8_{18}$ by replacing each crossing with a vertex, see Fig.~\ref{gr818}, equals $C_{11}C_{12}C_{22}C_{21}F_{21}F_{13}F_{31}F_{23}$. By using {\em Mathematica} this monomial does not belong to the ideal $I_2$, where $\delta=1$. Moreover, the graph $G_{8_{18}}$ is $2$-irreducible. Therefore, the leading term of $[[K]](f)_2$ is $G_{8_{18}}$ and the knot $8_{18}$ is minimal, i.e., any equivalent knot contains the knot $8_{18}$ after smoothing some its crossings.

 \begin{figure}
  \centering\includegraphics[width=80pt]{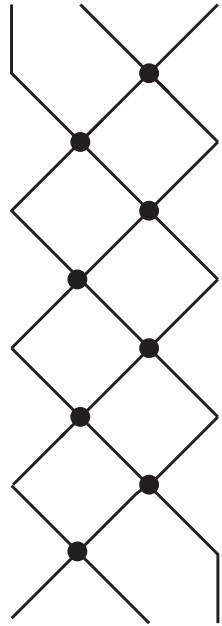}
  \caption{The leading term $G_{8_{18}}$}\label{gr818}
 \end{figure}

2) Let $L$ be the Borromean rings, see Fig.~\ref{bor}.

 \begin{figure}
  \centering\includegraphics[width=150pt]{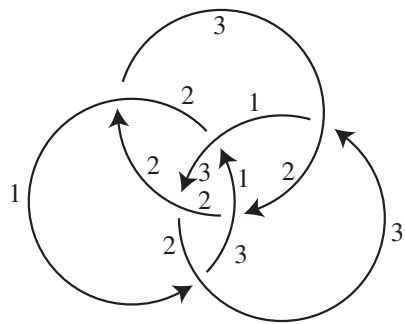}
  \caption{The Borromean rings}\label{bor}
 \end{figure}

Let us consider the biquanlde $X=\{1,2,3\}$, where two operations $\circ$ and $\ast$ are defined by the following matrices:
 $$
\circ=\left(\begin{array}{ccc}3&3&3\\1&1&1\\2&2&2\end{array}\right)\qquad\hbox{and}\qquad
\ast=\left(\begin{array}{ccc}3&3&3\\1&1&1\\2&2&2\end{array}\right),
 $$
i.e., $1\circ1=1\circ2=1\circ3=1\ast1=1\ast2=1\ast3=3$, $2\circ1=2\circ2=2\circ3=2\ast1=2\ast2=2\ast3=1$, $3\circ1=3\circ2=3\circ3=3\ast1=3\ast2=3\ast3=2$. By using this biquandle we colour our diagram by $f$ as shown in Fig.~\ref{bor}.

Let us find a leading term of the parity-biquandle bracket $[[L]](f)_1$. It is easy to see that the coefficient before the graph $G_{bor}$ obtained from the link $L$ by replacing each crossing with a vertex, see Fig.~\ref{bor1}, equals $F_{13}F_{21}C_{21}F_{32}C_{32}C_{13}$. By using {\em Mathematica} this monomial does not belong to the ideal $I_1$, where $\delta=1$. Moreover, the graph $G_{bor}$ is $2$-irreducible. Therefore, the leading term of $[[L]](f)_1$ is $G_{bor}$ and the link $L$ is minimal, i.e., any equivalent link contains the Borromean rings after smoothing some its crossings.

 \begin{figure}
  \centering\includegraphics[width=150pt]{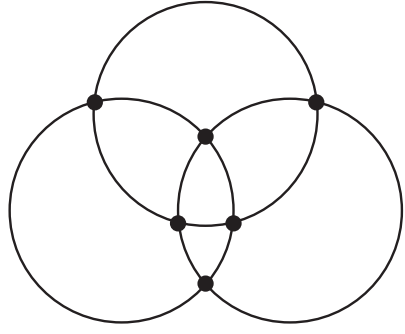}
  \caption{The leading term $G_{bor}$}\label{bor1}
 \end{figure}

 \section*{Acknowledgments}

The authors are grateful to I.\,M.~Nikonov, D.\,A.~Fedoseev and S.~Kim for his interest to the work.

 \end{document}